\documentclass{amsart}
\usepackage{cases}
\usepackage{amsfonts}
\usepackage{graphicx}
\usepackage{stmaryrd}
\usepackage{amssymb}
\usepackage{mathrsfs}
\usepackage{amsthm}
\usepackage{amsmath}
\usepackage{bbm}

\newtheorem{thm}{Theorem}[section]

\newtheorem{lem}[thm]{Lemma}

\newtheorem*{remark}{Remark}
\newtheorem*{thma}{Theorem A}

\numberwithin{equation}{section}

\newcommand{\ImPt}{\mathrm{Im}\,}

\newcommand{\calU}{\mathcal{U}}

\newcommand{\bfi}{\mathbf{i}}

\newcommand{\bfrho}{\boldsymbol{\rho}}
\newcommand{\calB}{\mathcal{B}}

\newcommand{\bbC}{\mathbb{C}}
\newcommand{\bbR}{\mathbb{R}}
\newcommand{\bbH}{\mathbb{H}}

\newcommand{\bfdelta}{\boldsymbol{\delta}}

\begin{document}

\title[$L^p-L^q$ boundedness of Bergman-type operators]{$L^p-L^q$ boundedness of Bergman-type operators over the Siegel upper half-space}

\thanks{The first author was supported by the National Natural Science
Foundation of China grants 11571333, 11471301.}

\author[C. Liu]{Congwen Liu}

\address{School of Mathematical Sciences,
         University of Science and Technology of China,
         Hefei, Anhui 230026,
         People's Republic of China.\\
and \\
Wu Wen-Tsun Key Laboratory of Mathematics,
USTC, Chinese Academy of Sciences}

\email{cwliu@ustc.edu.cn}

\author[J. Si]{Jiajia Si}
\address{School of Mathematical Sciences,
         University of Science and Technology of China,
         Hefei, Anhui 230026,
         People's Republic of China.}
\email{sijiajia@mail.ustc.edu.cn}

\author[P. Hu]{Pengyan Hu}
\address{Colledge of Mathematics and Statistics,
         Shenzhen University, Shenzhen, Guangdong 518060,
         People's Republic of China.}

\email{pyhu@szu.edu.cn}

\begin{abstract}
We characterize the $L^p-L^q$ boundedness of Bergman-type operators over the Siegel upper half-space.
This extends a recent result of Cheng et. al. (Trans. Amer. Math. Soc. 369:8643--8662, 2017) to higher dimensions.
\end{abstract}


\maketitle

\section{Introduction}
\label{intro}

Let $\bbC_{+}:=\{z\in \bbC: \ImPt z >0\}$ be the upper half-plane.
For $\alpha>0$, consider the integral operator
\begin{equation*}
T_{\alpha} f(z) = \int\limits_{\bbC_{+}} \frac {f(w)} {(z-\bar{w})^{\alpha}} dA(w), \quad z\in \bbC_{+}.
\end{equation*}
where $dA$ is the Lebesgue measure on $\bbC_{+}$.

Very recently, Cheng, Fang, Wang and Yu \cite{CFWY17} characterized the $L^p-L^q$ boundedness of $T_{\alpha}$ as follows.

\begin{thma}[{\cite[Theorem 5]{CFWY17}}]
Let $\alpha>0$ and $1\leq  p,q \leq \infty$.
\begin{enumerate}
\item[(i)]
If $\alpha>2$ then $T_{\alpha}:L^{p}(\bbC_{+})\to L^{q}(\bbC_{+})$ is unbounded for any $1\leq  p,q \leq \infty$.

\item[(ii)]
If $0<\alpha \leq 2$, then $T_{\alpha}:L^{p}(\bbC_{+})\to L^{q}(\bbC_{+})$ is bounded if and only if $p,q$ satisfy
\begin{equation*}
1<p< \frac {2}{2-\alpha} \quad \text{and} \quad  \frac{1}{q} = \frac{1}{p} + \frac{\alpha}{2} - 1.
\end{equation*}
\end{enumerate}
\end{thma}

The purpose of this note is to extend the above result to the several complex variables setting.

We fix a positive integer $n$ throughout this paper and let $\bbC^n = \bbC\times \cdots \times \bbC$
denote the $n$-dimensional complex Euclidean space.
For $z\in \bbC^n$, we use the notation
\[
z=(z^{\prime},z_{n}), \quad \text{where } z^{\prime}=(z_1,\ldots,z_{n-1})\in \bbC^{n-1} \text{ and } z_{n}\in \bbC^1.
\]
The Siegel upper half-space in $\bbC^n$ is the set
\[
\calU := \left\{ z\in \bbC^n: \ImPt z_{n} > |z^{\prime}|^2 \right\}.
\]
%Note that when $n=1$, $\calU=\bbC_{+}:=\{z\in \bbC: \ImPt z>0\}$, the classical upper half-plane.
This domain is biholomorphically equivalent to the unit ball of $\bbC^n$, and its boundary
$b\calU:= \left\{ z\in \bbC^n: \ImPt z_{n} = |z^{\prime}|^2 \right\}$ is
the standard representation of the Heisenberg group $\bbH^{n-1}$.
See \cite[Chapters 9--10]{Kra09} and \cite[Chapter XII]{Ste93}.

For real parameter $\alpha>0$, we consider the integral operator
\[
T_{\alpha} f(z) := \int\limits_{\calU} \frac {f(w)} { \bfrho(z,w)^{\alpha}} dV(w), \quad z\in \calU,
\]
where
\begin{equation*}
\bfrho(z,w) ~:=~ \frac {i}{2} (\bar{w}_{n}-z_{n})
- \langle z^{\prime}, w^{\prime} \rangle
\end{equation*}
with
\[
\langle z^{\prime},w^{\prime} \rangle := z_1\bar{w}_1 + \cdots + z_{n-1}\bar{w}_{n-1},
\]
and $dV$ is the Lebesgue measure on $\bbC^{n}$.
These operators are modelled on the Bergman projection on $\calU$.
Recall that the Bergman projection $P$ on $\calU$ is given by
\begin{equation*}
P f(z) ~=~ \frac {n!}{4\pi^{n}}\, T_{n+1} f(z) ~=~ \frac {n!}{4\pi^{n}}\, \int\limits_{\calU} \frac {f(w)} {\bfrho(z,w)^{n+1}} dV(w),
\quad z\in \calU.
\end{equation*}

Our main result is the following

\begin{thm}\label{thm:main}
Let $\alpha>0$ and $1\leq  p,q \leq \infty$.
\begin{enumerate}
\item[(i)]
If $\alpha>n+1$ then $T_{\alpha}$ is unbounded from $L^{p}(\mathcal{U})$ to $L^{q}(\mathcal{U})$ for any $1\leq  p,q \leq \infty$.

\item[(ii)]
If $0<\alpha \leq n+1$, then $T_{\alpha}:L^{p}(\mathcal{U})\to L^{q}(\mathcal{U})$ is bounded if and only if $p,q$ satisfy
\begin{equation}\label{eqn:nscondn}
1<p< \frac {n+1}{n+1-\alpha} \quad \text{and} \quad  \frac{1}{q} = \frac{1}{p} + \frac{\alpha}{n+1} - 1.
\end{equation}
\end{enumerate}
\end{thm}

\begin{remark}
Our proof also shows that Theorem \ref{thm:main} remains true if $T_{\alpha}$ is replaced by the integral operator
\[
S_{\alpha} f(z) := \int\limits_{\calU} \frac {f(w)} { |\bfrho(z,w)|^{\alpha}} dV(w), \quad z\in \calU.
\]
\end{remark}

\section{Preliminaries}

We begin by recalling the definition of the Heisenberg group and some basic facts which can be found in
\cite[Chapter XII]{Ste93}. The Heisenberg group $\bbH^{n-1}$ is the set
\[
\bbC^{n-1} \times \bbR = \{ [\zeta,t]: \zeta\in \bbC^{n-1}, t\in \bbR\},
\]
endowed with the group operation
\[
[\zeta,t]\cdot [\eta,s]=[\zeta+\eta, t+s+2 \mathrm{Im}\langle\zeta,\eta\rangle],
\]
where $\langle\zeta,\eta\rangle := \zeta_1 \overline{\eta}_1 + \cdots + \zeta_{n-1}\overline{\eta}_{n-1}$.
%Here we shall use square brackets [\, ] for elements of the Heisenberg group to distinguish them from
%points in $\bbC^{n}$, for which parentheses (\,) are used. The identity element is $[0,0]$ and the inverse of
%$[\zeta,t]$ is $[\zeta,t]^{-1}=[-\zeta,-t]$.
To each element $h=[\zeta,t]$ of $\bbH^{n-1}$, we associate the following (holomorphic) affine self-mapping of
$\calU$:
\begin{equation}\label{eqn:groupaction}
h:\; (z^{\prime},z_{n}) \longmapsto (z^{\prime}+\zeta, z_{n}+t+ 2i \langle z^{\prime}, \zeta\rangle + i|\zeta|^2).
\end{equation}
These mappings are simply transitive on the boundary $b\calU$ of $\calU$,
so we can identify the Heisenberg group with $b\calU$ via its action on the origin
\[
\bbH^{n-1} \ni [\zeta,t] ~\longmapsto~ (\zeta, t+i|\zeta|^2) \in b\calU.
\]
Also, it is easy to check that
\begin{equation}\label{eqn:h-inv}
\bfrho(h(z),h(w))=\bfrho(z,w)
\end{equation}
for any $z,w\in \calU$ and any $h\in \bbH^{n-1}$.
%Also, $J_{\bbR} h(z) \equiv 1$, where $J_{\bbR} h(z)$ is the real Jacobian of $h$ at $z\in \calU$.

\begin{lem}\label{lem:volume}
For any fixed $w\in \calU$ and any $R>0$, we have
\begin{equation}\label{eqn:equi-volume}
\big|\{z\in\mathcal{U}:|\bfrho(z,w)|<R\}\big| ~\leq~ \frac {2^{n+1} \pi^n}{(n-1)!} R^{n+1}.
\end{equation}
Here and in the sequel, $|E|=V(E)$ denotes the Lebesgue measure of $E\subset \bbC^n$.

\end{lem}

\begin{proof}
We first note that
\begin{equation}\label{eqn:equi-volume1}
\big|\{z\in \mathcal{U}:|\bfrho(z,w)|<R\}\big| ~=~ \big|\{z\in \mathcal{U}:|\bfrho(z, h(w))|<R\}\big|
\end{equation}
for any $w\in \calU$ and any $h\in\mathbb{H}^{n-1}$.
Indeed, by \eqref{eqn:h-inv},
\begin{align*}
\int\limits_{\substack{|\bfrho(z,w)|<R,\\ \ImPt z_{n} > |z^{\prime}|^{2}}}  dV(z)
~=~ \int\limits_{\substack{|\bfrho(h(z),h(w))|<R,\\ \ImPt z_{n} > |z^{\prime}|^{2}}} dV(z)
~=~ \int\limits_{\substack{|\bfrho(z,h(w))|<R,\\ \ImPt z_{n} > |z^{\prime}|^{2}}} dV(z),
\end{align*}
where the last equality follows by the change of variables $z\mapsto h(z)$ in the integral.

Now let $h=[-w^{\prime},0]$, then $h(w)=(0^{\prime}, w_{n}-i|w^{\prime}|^{2})$ and
\[
\bfrho(z,h(w))=\frac{i}{2} \left(\overline{w}_{n}+i|w'|^{2}-z_{n} \right).
\]
Also, the inequalities $\left|z_{n}-\overline{w}_{n}-i|w^{\prime}|^{2}\right| < 2R$ and $\ImPt z_{n} > |z^{\prime}|^{2}$
imply that $|z^{\prime}|^2 \leq 2R$. Hence
\begin{align*}
\big|\{z\in \calU: |\bfrho(z, h(w))|<R \}\big| ~=~&
\int\limits_{\substack{\left|z_{n}-\overline{w}_{n}-i|w^{\prime}|^{2}\right|< 2R,\\ \ImPt z_{n} > |z^{\prime}|^{2}}} dV(z)\\
\leq~& \int\limits_{\substack{\left|z_{n}-\overline{w}_{n}-i|w^{\prime}|^{2}\right|< 2R,\\ |z^{\prime}| \leq \sqrt{2R}}} dV(z)\\
~=~& \pi (2R)^2 \cdot \frac {\pi^{n-1}}{(n-1)!} (2R)^{n-1}
~=~ \frac {2^{n+1} \pi^n}{(n-1)!} R^{n+1}.
\end{align*}
Together with \eqref{eqn:equi-volume1}, this completes the proof.
\end{proof}

For $0 < p < \infty$, the weak-$L^p$ space $L^{p,\infty}(\calU)$ is defined as the set of
all measurable functions $f$ such that
\[
\| f \|_{L^{p,\infty}} := \sup_{\lambda>0} \, \lambda\cdot |\{z\in \calU: |f(z)|>\lambda\}|^{1/p}
\]
is finite.

We need the following variant of Schur's test, which is a special case of Lemma 1.11.17 in \cite[p.181]{Tao10}.

\begin{lem}[Weak-type Schur's test] \label{lem:variantofschur}
Let $1<p<q<\infty$ and $1<r<\infty$ be such that $\frac{1}{q}=\frac{1}{p}+\frac{1}{r}-1$.
Suppose that $Q(z,w)$ is a measurable function on $\calU \times \calU$ and $T$
is the associated integral operator
\[
T f(z) = \int\limits_{\calU} Q(z,w)f(w)\ dV(w), \quad z\in \calU.
\]
If there exists a positive constant $C$ such that
\[
\|Q(\cdot,w)\|_{L^{r,\infty}} \leq C
\]
for almost every $w\in \calU$ and
\[
\|Q(z,\cdot)\|_{L^{r,\infty}}\leq C
\]
for almost every $z\in \calU$, then $T$ is bounded from $L^{p}(\calU)$ to $L^{q}(\calU)$.
\end{lem}

We denote by $H^{\infty}(\calU)$ the space of bounded holomorphic functions on $\calU$,
and by $L_a^2(\calU)$ the closed subspace of $L^2(\calU)$ consisting of
holomorphic functions on $\calU$. The orthogonal projection from $L^2(\calU)$ onto $L_a^2(\calU)$,
known as the Bergman projection, can be expressed as an integral operator:
\[
P f(z) = \int\limits_{\calU} f(w) K(z,w) dV(w),
\]
with the Bergman kernel
\begin{equation}\label{eqn:bergknl}
K(z,w) ~:=~ \frac {n!}{4\pi^{n}}\, \frac {1}{\bfrho(z,w)^{n+1}}, \quad (z,w)\in \calU\times \calU.
\end{equation}
For $z\in \calU$, we put $K_z(\cdot):= K(\cdot,z)$ and $k_z:=K_z/\|K_z\|_2$.
The Berezin transform on $\calU$ is given by
\[
\mathcal{B} f(z) ~:=~ \langle fk_{z},k_{z}\rangle
~=~ \frac {n!}{4\pi^{n}} \int\limits_{\calU} f(w) \frac{\bfrho(z)^{n+1}}{|\bfrho(z,w)|^{2n+2}}\ dV(w), \quad z\in \calU,
\]
where $\bfrho(z):=\bfrho(z,z)= \ImPt z_n - |z^{\prime}|^2$.
%where $c_{n}=\frac{n!}{4\pi^{n}}$, which plays important roles in Berezin's theory of quantization as well as in the theory of Toeplitz operators. We display a simple Lemma here as following.

\begin{lem}\label{lem:berezin}
If $f\in H^{\infty}(\calU)$ then $\calB  f = f$.
\end{lem}

\begin{proof}
Since $f\in H^{\infty}(\calU)$, for each fixed $z\in \calU$, $fk_z\in L_a^2(\calU)$. By the reproducing property of $K_z$,
\[
\calB f(z) ~=~ \frac {1}{\sqrt{K(z,z)}} \langle fk_{z},K_{z}\rangle ~=~ \frac{1}{\sqrt{K(z,z)}} f(z) k_{z}(z) ~=~ f(z).
\]
\end{proof}

We end this section by recalling two formulas from \cite{LLHZ17}.

\begin{lem}[{\cite[Key Lemma]{LLHZ17}}]\label{lem:keylemma}
Suppose that $r,\,s>0$, $t>-1$ and $r+s-t>n+1$. Then
\begin{equation}\label{eqn:keylem2}
\int\limits_{\calU}  \frac {\bfrho(w)^{t}} {\bfrho(z,w)^{r} \bfrho(w,u)^{s}} dV(w)
~=~ \frac {C_1(r,s,t)} {\bfrho(z,u)^{r+s-t-n-1}}
\end{equation}
holds for all $z, u\in \calU$, where
\begin{equation}\label{eqn:const}
C_1(r,s,t) ~:=~  \frac {4\pi^{n} \Gamma(1+t)\Gamma(r+s-t-n-1)}{\Gamma(r)
\Gamma(s)}.
\end{equation}
\end{lem}

\begin{lem}[{\cite[Lemma 5]{LLHZ17}}]\label{lem:keylem2}
Let $s,t\in \bbR$. Then we have
\begin{equation}\label{eqn:keylem}
\int\limits_{\calU} \frac {\bfrho(w)^{t}} {|\bfrho(z,w)|^{s}} dV(w) ~=~
\begin{cases}
\dfrac {C_2(s,t)} {\bfrho(z)^{s-t-n-1}}, &
\text{ if } t>-1 \text{ and } s-t>n+1\\[12pt]
+\infty, &  otherwise
\end{cases}
\end{equation}
for all $z\in \calU$, where
\[
C_2(s,t):=\frac {4 \pi^{n} \Gamma(1+t) \Gamma(s-t-n-1)} {\Gamma^2\left(s/2\right)}.
\]
\end{lem}

\section{Proof of Theorem \ref{thm:main}: Part (i)}

We begin with the following lemma.

\begin{lem}\label{lem:necty1}
If  $T_{\alpha}:L^{p}(\mathcal{U})\to L^{q}(\mathcal{U})$ is bounded, then $p$ and $q$
must be related by
\begin{equation}\label{eqn:necty1}
\frac{1}{q} = \frac{1}{p} + \frac{\alpha}{n+1} -1.
\end{equation}
\end{lem}

\begin{proof}

Suppose that $T_{\alpha}:L^{p}(\mathcal{U})\to L^{q}(\mathcal{U})$ is bounded, that is,
there is a positive constant $C=C(p,q,n,\alpha)$ such that
\begin{equation}\label{eqn:pqbddnss}
\|T_{\alpha}f \|_q \leq C \|f\|_p
\end{equation}
for all $f\in L^p(\calU)$.

Fix a function $f\in L^p(\calU)$, say $f(z)= |\bfrho(z,\bfi)|^{-n-2}$, where $\bfi=(0^{\prime}, i)$. For $t>0$, we define the dilation $t\circ z$ by
\[
t\circ z = (tz^{\prime}, t^2 z_n), \quad z = (z^{\prime}, z_n) \in \calU.
\]
It is obvious that the dilations map $\calU$ to $\calU$.
Now we consider the dilations $\bfdelta^{t} (f)$ of $f$ given by
\[
\bfdelta^{t} (f)(z) := f(t\circ z), \quad z\in \calU.
\]
It is easy to verify that
\begin{equation}\label{eqn:p-norm-of-dilation}
\|\bfdelta^{t} (f) \|_p = t^{- \frac {2(n+1)}{p}} \|f\|_p.
\end{equation}
Note that $\bfrho(t\circ z, t\circ w) = t^{2}\bfrho(z,w)$ holds for any $z,w\in \calU$ and any $t>0$. Making the change of variables
$u = t\circ w$ in the integral defining $T_{\alpha} (\bfdelta^t (f))$, we see that
\begin{align*}
T_{\alpha} (\bfdelta^t (f)) (z) ~=~& \int\limits_{\calU} \frac {f(u)} {t^{-2\alpha} \bfrho(t\circ z, u)^{\alpha}} t^{-2(n+1)}\ dV(u)\\
=~& t^{2(\alpha-n-1)} \bfdelta^t (T_{\alpha} f)(z),
\end{align*}
and hence
\begin{equation}\label{eqn:q-norm-of-dilation}
\| T_{\alpha} (\bfdelta^t (f) )\|_q  = t^{2(\alpha-n-1 - \frac {n+1}{q}) }  \| T_{\alpha} f \|_q.
\end{equation}
Replacing $f$ by $\bfdelta^{t} (f)$ in \eqref{eqn:pqbddnss} and using \eqref{eqn:p-norm-of-dilation}
and \eqref{eqn:q-norm-of-dilation}, we obtain
\begin{equation}\label{eqn:rescpqbddnss}
t^{2(\alpha-n-1 - \frac {n+1}{q}) } \|T_{\alpha}f\|_q ~\leq~ C\,  t^{-\frac {2(n+1)}{p} } \|f\|_p.
\end{equation}
Suppose now that $\frac{1}{q} < \frac{1}{p} + \frac{\alpha}{n+1} -1$. We can write \eqref{eqn:rescpqbddnss} as
\[
\|T_{\alpha}f\|_q ~\leq~ C\,  t^{2(n+1)(\frac {1}{q} - \frac {1}{p} - \frac {\alpha}{n+1} +1)} \|f\|_p
\]
and let $t\to \infty$ to obtain that $T_{\alpha} f =0$, a contradiction. Similarly,
if $\frac{1}{q} > \frac{1}{p} + \frac{\alpha}{n+1} -1$, we could write \eqref{eqn:rescpqbddnss} as
\[
t^{2(n+1)(\frac {1}{p} + \frac {\alpha}{n+1}-1-\frac {1}{q} )} \|T_{\alpha}f\|_q
~\leq~ C  \|f\|_p
\]
and let $t\to 0$ to obtain that $\| f\|_p = \infty$, again a contradiction. It follows that
\eqref{eqn:necty1} must necessarily hold.
\end{proof}

Now we turn to the proof of Part (i) of Theorem \ref{thm:main}.

We argue by contradiction. Suppose that $\alpha > n+1$ and $T_{\alpha}:L^{p}(\mathcal{U})\to L^{q}(\mathcal{U})$ is bounded.
By Lemma \ref{lem:necty1}, $p$ and $q$
must be related by
\begin{equation}\label{eqn:necty2}
\frac{1}{q} = \frac{1}{p} + \frac{\alpha}{n+1} -1.
\end{equation}

Let $N$ be a positive integer such that $N > n+1$ and consider the function
\[
f_N (z) ~:=~ \bfrho(z,\mathbf{i})^{-N}, \quad z\in\mathcal{U}.
\]
By Lemma \ref{lem:keylem2}, we see that $f_N\in L^{p}(\mathcal{U})$ and
\begin{equation*}
\| f_N \|_p ~=~ \left\{ \frac {4\pi^n \Gamma(pN-n-1)} {\Gamma^2 ( pN/2) } \right\}^{1/p}.
\end{equation*}
Using Legendre's duplication formula (see \cite[p.26, (2.3.1)]{BW10})
\begin{equation}\label{eqn:Legendre}
\Gamma(2x) = \frac {2^{2x-1}}{\sqrt{\pi}} \Gamma(x)\Gamma\left(x+\frac {1}{2} \right),
\quad x\neq 0, -1, -2, \ldots
\end{equation}
and the asymptotic formula for the Gamma function (\cite[p.22, (2.1.9)]{BW10})
\begin{equation}\label{eqn:Stirling}
\frac {\Gamma(x+a)}{\Gamma(x)} = x^a \left[ 1+ O\left(x^{-1}\right)\right], \quad \text{as } x\to +\infty,
\end{equation}
we see that
\begin{equation}
\| f_N \|_p ~\sim~ 2^N \cdot N^{-\frac {n}{p} - \frac {1}{2p}}, \quad \text{as } N\to \infty.
\end{equation}
Also, by Lemma \ref{lem:keylemma}, we have
\begin{align*}
(T_{\alpha} f_N)(z) ~=~ &\int\limits_{\mathcal{U}} \frac{ dV(w)} {\bfrho(z,w)^{\alpha}\bfrho(w,\mathbf{i})^{N}}\\\
=~ & \frac {4\pi^n \Gamma(N+\alpha-n-1)} {\Gamma(N)\Gamma(\alpha)} \,  \bfrho(z,\bfi)^{n+1-\alpha-N}.
\end{align*}
Thus, by Lemma \ref{lem:keylem2}, we obtain
\begin{equation*}
\|T_{\alpha}f_N \|_q ~=~ \frac {4\pi^n \Gamma(N+\alpha-n-1)} {\Gamma(N)\Gamma(\alpha)}
\, \left\{ \frac {4\pi^n \Gamma(q(N+\alpha-n-1)-n-1) }{ \Gamma^2(q(N+\alpha-n-1)/2)}\right\}^{1/q}.
\end{equation*}
Again, by \eqref{eqn:Legendre} and \eqref{eqn:Stirling}, we get
\begin{equation}
\| T_{\alpha} f_N \|_p ~\sim~ 2^N \cdot N^{-\frac {n}{q} - \frac {1}{2q} + \alpha - n-1}, \quad \text{as } N\to \infty.
\end{equation}
Since there exists a positive constant $C$ such that $\|T_{\alpha} f_N\|_q \leq C \|f_N \|_p$ for all $N > n+1$,
we can find another positive constant $C^{\prime}$, independent of $N$, such that
\begin{equation}\label{asym-est1}
2^N \cdot N^{-\frac {n}{q} - \frac {1}{2q} + \alpha - n-1} ~\leq~ C^{\prime} \cdot 2^N \cdot N^{-\frac {n}{p} - \frac {1}{2p}}
\end{equation}
for all $N>n+1$. Keeping \eqref{eqn:necty2} in mind, we see that
\[
-\frac {n}{q} - \frac {1}{2q} + \alpha - n-1 + \frac {n}{p} + \frac {1}{2p} ~=~ \frac {1}{2} \left( \frac {\alpha}{n+1} -1\right),
\]
and hence \eqref{asym-est1} can be rewritten as
\[
N^{\frac {1}{2}(\frac{\alpha}{n+1} -1)} ~\leq~ C^{\prime}
\]
for all $N>n+1$. But this is impossible, since $\alpha>n+1$.

\section{Proof of Theorem \ref{thm:main}: Part (ii)}

\subsection{Necessity}

As we have already shown in Lemma \ref{lem:necty1}, if $T_{\alpha}$ is bounded from $L^{p}(\mathcal{U})$ to $L^{q}(\mathcal{U})$
then $p$ and $q$ satisfy
\begin{equation*}
\frac{1}{q} = \frac{1}{p} + \frac{\alpha}{n+1} -1.
\end{equation*}
It remains to show that $T_{\alpha}$ is unbounded in the endpoint cases $(p,q)=(1,\frac {n+1}{\alpha})$ and
$(p,q)=(\frac {n+1}{n+1-\alpha}, \infty)$. We only consider the former case,
since, once this case is done, the other case follows by duality.
%Note that $T_{\alpha}$ is bounded from $L^{\frac {n+1}{n+1-\alpha}}(\mathcal{U})$ to $L^{\infty}(\mathcal{U})$
%if and only if the adjoint $T_{\alpha}^{\ast}$ ($= T_{\alpha}$) is bounded from $L^{1}(\mathcal{U})$ to
%$L^{\frac {n+1}{\alpha}}(\mathcal{U})$.

Consider the function
\[
f(z) ~:=~ \frac{1} {|\rho(z,\bfi)|^{2n+2}}, \quad z\in \calU.
\]
Then $f\in L^{1}(\calU)$, by Lemma \ref{lem:keylem2}.

It is clear that, for any fixed $z\in \calU$, the function $g_z(\cdot):=\bfrho(\cdot, z)^{-\alpha} \in H^{\infty}$.
Hence $\calB g_z = g_z$, in view of Lemma \ref{lem:berezin}. In particular, $\calB g_z(\bfi) = g_z(\bfi)$, that is,
\begin{equation}\label{eqn:bfequalsf}
\frac {n!}{4\pi^n} \int\limits_{\calU} \bfrho(w, z)^{-\alpha} \frac{\bfrho(\bfi,\bfi)^{n+1}}
{|\bfrho(\bfi,w)|^{2n+2}}\ dV(w) = \bfrho(\bfi, z)^{-\alpha}.
\end{equation}
Taking the complex conjugate of both sides of \eqref{eqn:bfequalsf}, we obtain
\[
(T_{\alpha} f)(z) ~=~ \frac {4\pi^n}{n!}\, \overline {\calB g_z(\bfi)} ~=~ \frac {4\pi^n}{n!}\, \overline {g_z(\bfi)}
~=~ \frac {4\pi^n}{n!}\, \bfrho(z,\bfi)^{-\alpha}, \quad z\in \calU.
\]
Hence, according to Lemma \ref{lem:keylem2},
\[
\| T_{\alpha} f\|_{\frac {n+1}{\alpha}}^{\frac {n+1}{\alpha}} ~=~ \left(\frac {4\pi^n}{n!}\right)^{\frac {n+1}{\alpha}} \,
\int\limits_{\calU} \frac{dV(z)} {|\rho(z,\bfi)|^{n+1}}
~=~ +\infty.
\]
This show that $T_{\alpha}$ does not send $L^1(\calU)$ into $L^{\frac {n+1}{\alpha}}(\calU)$.

\subsection{Sufficiency}

The case $\alpha=n+1$ is well-known (see for instance \cite[Lemma 2.8]{CR80}).

Suppose now that $0< \alpha < n+1$ and \eqref{eqn:nscondn} hold. Put $Q_{\alpha}(z,w)= \bfrho(z,w)^{-\alpha}$. For any fixed $w\in \mathcal{U}$,
by Lemma \ref{lem:volume}, we have
\begin{eqnarray*}
\| Q_{\alpha}(\cdot, w)\|_{L^{\frac{n+1}{\alpha},\infty}}
&=& \sup_{\lambda>0}\, \lambda \cdot \big| \{z\in \mathcal{U}:|Q_{\alpha}(z,w)|>\lambda\}\big|^{\frac{\alpha}{n+1}} \\
&=& \sup_{\lambda>0}\, \lambda \cdot \left| \left\{z\in \mathcal{U}: |\bfrho(z,w)| < \lambda^{-\frac {1}{\alpha}} \right\}\right|^{\frac{\alpha}{n+1}} \\
&\leq&  \sup_{\lambda>0}\, \lambda \cdot \left(\frac {2^{n+1} \pi^n}{(n-1)!}\lambda^{-\frac{n+1}{\alpha}}\right)^{\frac{\alpha}{n+1}}\\
&=& \left(\frac {2^{n+1} \pi^n}{(n-1)!}\right)^{\frac{\alpha}{n+1}}
\end{eqnarray*}
for all $w\in \calU$. By symmetry,
\[
\| Q_{\alpha}(z,\cdot)\|_{L^{\frac{n+1}{\alpha},\infty}} ~\leq~ \left(\frac {2^{n+1} \pi^n}{(n-1)!}\right)^{\frac{\alpha}{n+1}}
\]
for all $z\in \calU$. Therefore, $T_{\alpha}$ is bounded from $L^{p}(\calU)$ to $L^{q}(\calU)$, by Lemma \ref{lem:variantofschur}.

\end{document}